\documentclass[11pt]{article}

\usepackage{amsfonts}
\usepackage{graphics}
\usepackage{amssymb}

\newtheorem{theorem}{Theorem}[section]
\newtheorem{lemma}{Lemma}[section]
\newtheorem{definition}{Definition}[section]
\newtheorem{proposition}[theorem]{Proposition}
\newtheorem{corollary}[theorem]{Corollary}
\newtheorem{example}{Example}[section]
\newtheorem{remark}{Remark}[section]

\newcommand{\eop }{ \hfill $\Box$ }

\begin{document}

\begin{center}
{\Large Martingales on Principal Fiber Bundles\\}

\end{center}

\vspace{0.3cm}

\begin{center}
{\large Pedro Catuogno}\\
\textit{Departamento de
 Matem\'{a}tica, Universidade Estadual de Campinas,\\ 13.081-970 -
 Campinas - SP, Brazil. e-mail: pedrojc@ime.unicamp.br}\\
\vspace{0.5cm}
{\large Sim\~ao N. Stelmastchuk}\\

\textit{Universidade Federal do Paraná,\\ 86900-000 -
 Jandaia do Sul- PR, Brazil. e-mail:simnaos@ufpr.br}
\end{center}

\vspace{0.3cm}

\begin{abstract}

	Let $P(M,G)$ be a principal fiber bundle, let $\omega$ be a connection form on $P(M,G)$, and consider a projectable connection $\nabla^{P}$ on $P(M,G)$. The aim of this work is to determine the $\nabla^{P}$-martingales in  $P(M,G)$. Our results allow establishing new characterizations of harmonic maps from Riemannian manifolds to principal fiber bundles.
\end{abstract}

\noindent {\bf Key words:} martingales; principal fiber bundles;
harmonic maps.

\vspace{0.3cm} \noindent {\bf MSC2000 subject classification:}
 53C43, 60H30, 60G48.

\bibliographystyle{apalike}
\section{Introduction}

\noindent This article is concerned with the characterization of
martingales in a principal fiber bundle $P(M,G)$ which is endowed with a
connection form $\omega$ and a $G$-invariant symmetric connection
$\nabla^{P}$. Following to N. Abe and K. Hasegawa \cite{abe}, we
says that $\nabla^{P}$ is projectable if there exists an unique
symmetric connection $\nabla^{M}$ on $M$ such that for any vector
fields $X$ and $Y$ on $M$
\[
\mathbf{h}\nabla^{P}_{X^h}Y^h = (\nabla^{M}_{X}Y)^h,
\]
where $\mathbf{h}$ denotes the horizontal projection which is
associated to $\omega$ and $-^h$ denotes the associated horizontal
lift. Let $A$ and $T$ be the fundamental tensors associated to
$\pi : P \rightarrow M$, see (\ref{T}) and (\ref{A}) below.

In this situation, we prove the following results:

\noindent 1) Let $Y$ be a continuous semimartingale with values in
$P$. Then $Y$ is a $\nabla^{P}$-martingale if and only if
\begin{equation}
\int \omega \;\delta Y - \frac{1}{2}\int (\nabla^{P}\omega)(dY,dY)
\end{equation}
is a local martingale and
\begin{equation}
\int \alpha \;d^{\nabla^{M}} \pi \circ Y + \frac{1}{2} \int \alpha
\circ \pi_* \circ (2A^S+T^S)(dY,dY)
\end{equation}
is a local martingale for all $\alpha \in \Gamma(T^*M)$.

\noindent 2) Let $N$ be a Riemannian manifold with metric $g$ and
$F:N \rightarrow P$ be a smooth map. Then $F$ is a harmonic map if
and only if
\begin{equation}
d^{*}F^{*} \omega = trF^{*}(\nabla^{P}\omega)
\end{equation}
and
\begin{equation}
\tau_{\pi \circ F} = -\mathrm{tr}~ \pi_* \circ (2A^S+T^S) \circ
F_* \otimes F_*.
\end{equation}

In \cite{cat-stel}, the present authors give similar
characterizations in the special case of the frame bundle.

The motivation for this work is to understand the
stochastic differential geometry of principal fiber bundles. In
fact, we are interested in study the rich interplay between
differential geometry and stochastic calculus.

This paper is organized in the following way: In section 2 we
review some fundamental facts on the differential geometry of
principal fiber bundles and stochastic calculus on manifolds. In
section 3 we prove our principal results. Finally, we apply our
results to study the stochastic differential geometry of
Kaluza-Klein theory.

\section{Preliminaries}

We begin by recalling some fundamental facts on the differential
geometry of principal fiber bundles and stochastic calculus on
manifolds. We shall use freely concepts and notations of S.
Kobayashi and N. Nomizu \cite{kobay}, M. Emery \cite{emery1} and
P. Meyer \cite{meyer1}.

Let $P(M,G)$ be a principal fiber bundle with projection $\pi: P
\rightarrow M$. Let us denote the right action of $G$ on $P$ by
$R_g (p)=pg$ for $p \in P$ and $g \in G$. A horizontal lift $H$ in
$P(M,G)$ is a smooth family of applications $H_{p}:T_{\pi(p)}M
\rightarrow T_{p}P$ such that $\pi_* \circ H_{p}=Id_{T_{\pi(p)}M}$
for all $p \in P$ and $(R_g)_*H_p=H_{pg}$ for all $p \in P$ and $g
\in G$. The horizontal lift $H$ determines a unique decomposition
of each tangent space $T_{p}P$  which is the direct sum of the vertical
subspace $V_{p}P = Ker(\pi_{*}(p))$ and the horizontal subspace
$H_{p}P = Im(H_p)$ at $p \in P$. This decomposition naturally
defines the horizontal lifts of $X \in T_{\pi(p)}M$ as the unique
tangent vector $X^h=H_p(X) \in H_{p}P$ such that $\pi_{*}(X^h) =
X$. We denote by $\mathfrak{g}$ the Lie algebra of $G$. For $B \in
\mathfrak{g}$, the right action of $G$ into $P$ defines a
$1$-parameter transformation group on $P$ and induces a vector
field $B^{*}$ on $P$. We call $B^{*}$ the fundamental vector field
corresponding to $B$, which is a vertical vector field. For each
$p \in P$, the linear mapping $\sigma_p : \mathfrak{g} \rightarrow
V_{p}P$ defined by $\sigma_p(B)=B^*_p$ is an isomorphism.

Let us denote by $\mathbf{h}U_{p}$ and $\mathbf{v}U_{p}$
the horizontal and vertical parts of $U \in T_{p}P$, respectively.
The connection form $\omega : TP \rightarrow \mathfrak{g}$ is
defined by
\[
\omega(U_p) = B,
\]
where $\mathbf{v}U_p = B^{*}_p$ at $p \in P$.

We observe that the connection form is a
$\mathfrak{g}$-valued $1$-form on $P$ satisfying the following
conditions:
\begin{enumerate}
\item $\omega(B^*) = B$ for $B \in \mathfrak{g}$,

\item $R_g^*\omega=ad_{g^{-1}}\omega$ for $g \in G$,
\end{enumerate}
where $ad_{g}: \mathfrak{g} \rightarrow \mathfrak{g}$ is defined by
$ad_g(B)=(I_g)_*B$, $I_g$ being the inner automorphism of $G$,
$I_g(x)=gxg^{-1}$ for all $g \in G$.

Conversely, given a $\mathfrak{g}$-valued $1$-form $\omega$ on $P$, which satisfies
the above conditions, there is an unique horizontal
lift $H$ in $P$ that its connection form is $\omega$. For $p \in P$
and $X \in T_{\pi(p)}M$, we have that $H_p(X) \in T_pP$ is the
unique solution of $\pi_*(H_p(X))=X$ and $\omega_p(H_p(X))=0$.

The curvature form $\Omega$ of $\omega$ is the
$\mathfrak{g}$-valued $2$-form on $P$ defined by
\[
\Omega(U, V)=d\omega(\mathbf{h}U,\mathbf{h}V),
\]
where $U$ and $V$ are vector fields on $P$.


Throughout the paper we adopt the following convention: a connection on a manifold means a torsion free covariant
derivative operator on the tangent bundle.

\begin{definition} Let $\nabla^{P}$ be a connection on $P(M,G)$. We say that $\nabla^{P}$ is $G$-invariant if the right
translations $R_{g}$ are affine for all $g \in G$.

Let $\omega$ be a connection form on $P(M,G)$. A
$G$-invariant connection $\nabla^{P}$ on $P(M,G)$ is projectable
if $\mathbf{h}\nabla^{P}_{X^h}Y^h$ is projectable for all vector
fields $X$ and $Y$ on $M$.
\end{definition}

\begin{proposition}\label{prop1}
Let $P(M,G)$ be a principal fiber bundle, and consider a
$G$-invariant connection $\nabla^{P}$ on $P(M,G)$. Then $\nabla^{P}$ is
projectable if and only if there exist an unique connection
$\nabla^{M}$ on $M$ such that
\[
\mathbf{h}\nabla^{P}_{X^h}Y^h = (\nabla^{M}_{X}Y)^h
\]
for all vector fields $X$ and $Y$ on $M$.
\end{proposition}
\begin{proof}
Let $\nabla^{P}$ be a projectable connection. We define
$\nabla^{M}_{X}Y = \pi_{*} \nabla^{P}_{X^{h}}Y^{h}$, and so clearly
$\mathbf{h}\nabla^{P}_{X^{h}}Y^{h} = (\nabla^{M}_{X}Y)^{h}$. It
remains to prove that $\nabla^{M}$ is a connection. Since
$(fX)^{h} = (f \circ \pi)X^{h}$, for all $f \in C^{\infty}(M)$,
\begin{eqnarray*}
\nabla^{M}_{gX}(fY)
& = & \pi_{*} (\nabla^{P}_{(g \circ \pi)X^{h}}(f \circ \pi)Y^{h})\\
& = & \pi_{*}((g \circ \pi)X^{h}(f \circ \pi)Y^{h} + (g \circ \pi)(f \circ \pi)\nabla^{P}_{X^{h}}Y^{h})\\
& = & gX(f)Y + gf\nabla^{M}_{X}Y.
\end{eqnarray*}
The uniqueness follows from the definition. \eop
\end{proof}

Let $\nabla^{P}$ be a connection on $P(M,G)$, and take a connection form $\omega$ on $P(M,G)$. Following to B. O'Neill \cite{oneill}, we describe the geometrical quantities of our
interest in terms of the fundamental tensors $T$ and $A$. They are defined
by
\begin{equation}\label{T}
T_{U}V = \mathbf{h}\nabla^{P}_{\mathbf{v}U} \mathbf{v}V + \mathbf{v}\nabla^{P}_{\mathbf{v}U} \mathbf{h}V
\end{equation}
and
\begin{equation}\label{A}
A_{U}V = \mathbf{v}\nabla^{P}_{\mathbf{h}U} \mathbf{h}V +
\mathbf{h}\nabla^{P}_{\mathbf{h}U} \mathbf{v}V,
\end{equation}
for $U$ and $V$ vector fields on $P$.

We observe that the vanishing of $T$ means that the
fibers of $P(M,G)$ are totally geodesic.

\begin{lemma}
Let $P(M,G)$ be a principal fiber bundle, and consider a connection
form $\omega$ and a projectable connection $\nabla^{P}$ on $P(M,G)$. If $X$ and $Y$ are vector fields on $M$
and $B,C \in \mathfrak{g}$, then we have the following equations:
\begin{equation}\label{grupo1}
\left\{
\begin{array}{rcl}
T_{B^*}C^*  & = & T_{C^*}B^* \\
T_{B^*}X^h  & = &  \omega(\nabla^{P}_{X^h}B^*)^* \\
A_{X^h}Y^h  & = & -2\Omega(X^h, Y^h)^* + A_{Y^h}X^h \\
A_{X^h}B^*  & = &
\nabla^{P}_{X^h}B^*-\omega(\nabla^{P}_{X^h}B^*)^*+[X^h,B^*]
\end{array}
\right.
\end{equation}
and
\begin{equation}\label{grupo2}
\left\{
\begin{array}{rcl}
\nabla^{P}_{B^*}C^* & = & \widehat{\nabla}_{B^*}C^* + T_{B^*}C^*\\
\nabla^{P}_{B^*}X^h & = & \mathbf{h}\nabla^{P}_{B^*}X^h +
T_{B^*}X^h\\
\nabla^{P}_{X^h}B^* & = & \mathbf{v}\nabla^{P}_{X^h}B^*  + A_{X^h}B^*\\
\nabla^{P}_{X^h}Y^h & = & \mathbf{h}\nabla^{P}_{X^h}Y^h +
A_{X^h}Y^h,
\end{array}
\right.
\end{equation}
where $\widehat{\nabla}$ is the induced connection by $\nabla^P$
in the fibers.
\end{lemma}
\begin{proof}
An easy computation shows the equations (\ref{grupo2}). To prove
(\ref{grupo1}), we recall that $\mathbf{v}U= \omega(U)^{*}$ for
all $U \in TP$ and $\mathbf{v}[X^{h}, Y^{h}] = -2\Omega(X^h,
Y^h)^*$ for all $X,Y \in TM$. \eop
\end{proof}

From (\ref{grupo1}) it follows  that the horizontal
distribution $\{H_pP:p \in P\}$ is integrable if and only if
$A_{X^h}Y^h=A_{Y^h}X^h$ for all $X$ and $Y$ vector fields on $M$.

\begin{example}
Let $M$ be a differentiable manifold and choose a connection $\nabla$ 
on $M$. We consider the frame bundle $BM(M,GL(\mathbb{R}^n))$, which
is a principal fiber bundle with base $M$ and structure group
$GL(\mathbb{R}^n)$.

The canonical lift $\nabla^c$ and horizontal lift
$\nabla^h$ are projectable connections on $BM$ with the projection
$\nabla$. In the book of L. Cordero et al. \cite{cordero} we find
a survey of elementary properties of these connections. Let
$X$ and $Y$ be vector fields on $M$ and take $B,C \in \mathfrak{gl}(n,
\mathbb{R}^{n})$. The canonical lift $\nabla^{c}$ and
horizontal lift $\nabla^{H}$ are completely defined by the
relations:

\begin{equation}\label{canonical}
\left\{
\begin{array}{rcl}
\nabla_{B^{*}}^{c} C^{*} & = & (BC)^{*} \\
\nabla_{B^{*}}^{c} X^h & = & 0 \\
\nabla_{X^h}^{c} B^{*} & = & 0  \\
\nabla_{X^h}^{c} Y^h & = & (\nabla_{X} Y)^h + \gamma( R( - , X)Y)
\end{array}
\right.
\end{equation}
and
\begin{equation}\label{horizontal}
\left\{
\begin{array}{rcl}
\nabla_{B^{*}}^{H} C^{*}  & = & (BC)^{*}  \\
\nabla_{B^{*}}^{h} X^h & = & 0 \\
\nabla_{X^h}^{H} B^{*} & = & 0 \\
\nabla_{X^h}^{H} Y^h & = & (\nabla_{X} Y)^h,
\end{array}
\right.
\end{equation}
where $R$ is the curvature tensor of $\nabla$ and $\gamma S$ is
the vertical lift defined by $\gamma S (p) = (p^{-1}\circ S \circ p)^*(p)$ for $p \in BM$.

We observe that, in both cases, $T=0$ and $A_{X^h}B^*=0$,
so $\pi$ is affine.
\end{example}

\begin{remark}
In the 1920s T. Kaluza and O. Klein proposed the use of spaces of
dimension higher than four in order to unify general relativity
and what we now call Yang-Mills theories. From a mathematical view
point Kaluza-Klein theory is the differential geometry of a
principal fiber bundle with invariant Riemannian metric (see
\cite{coquereaux} for more details). In this context is
fundamental the following reduction theorem: Let $P(M,G)$ be a
principal fiber bundle and assume that $k$ is a $G$-invariant Riemannian
metric on $P$, namely $R_g$ is an isometry for all $g \in G$. Let
$\mathcal{M}_{ad}(\mathfrak{g})$ denote the set of metrics on
$\mathfrak{g}$ invariant by $ad_{g}$ for all $g \in G$. Then there
exist
\begin{enumerate}
\item $h$ a Riemannian metric on $M$,

\item $\omega$ a connection form on $P$,

\item $F: M \rightarrow \mathcal{M}_{ad}(\mathbf{g})$  a smooth
function
\end{enumerate}
such that
\begin{equation}\label{KK}
k(U,V)=h(\pi_*(U),\pi_*(V))+F \circ \pi (\omega(U), \omega(V))
\end{equation}
for all $U$ and $V$ vector fields on $P$.

Reciprocally, given a Riemannian metric $h$ on $M$, a
connection form $\omega$ on $P$, and a smooth function $F: M \rightarrow
\mathcal{M}_{ad}(\mathbf{g})$, we have that
(\ref{KK}) defines the unique $G$-invariant Riemannian metric $k$ on
$P$.

It is easy to check that $\pi : P \rightarrow M$ is a
Riemannian submersion and the connection form $\omega$ is
associated to the horizontal lift $H^k$. For $X \in TM$, $H^k(X)$
is completely determined by $\pi_*(H^k(X))=X$ and $H^k(X)$ is
orthogonal to $VP$.

It is clear that the Riemannian connection $\nabla^k$ associated to $k$ is a projectable connection with projection
$\nabla^h$, the Riemannian connection associated to $h$.

\end{remark}

\begin{example}
Let $P(M,G)$ be a principal fiber bundle with a connection
form $\omega$ on $P$. Assume that $k_0$ is an $ad(G)$-invariant metric on $\mathfrak{g}$ and that 
$h$ is a Riemannian metric on $M$. We consider the
$G$-invariant Riemannian metric $k$ on $P$ defined by
\begin{equation}\label{KK1}
k(U,V)=h(\pi_*(U),\pi_*(V))+k_0(\omega(U), \omega(V))
\end{equation}
for all $U$ and $V$ vector fields on $P$.

An easy computations shows that $\pi : P \rightarrow M$
is a Riemannian submersion with $T=0$ and
\begin{eqnarray*}\label{KK1}
A_{X^h}X^h & = 0 \\
A_{X^h}B^* & = & -\frac{1}{2}k_0 (B, \Omega(-,
X^h))^{\sharp}
\end{eqnarray*}
for all $X$ vector field on $M$ and $B \in \mathfrak{g}$ (see for
instance \cite{abe} and \cite{oneill}).
\end{example}

Let $(\Omega, (\mathcal{F}_t),\mathbb{P})$ be a filtered probability space, and let $M$ be a smooth manifold provided with a connection $\nabla$. Furthermore, assume that $X$ is a continuous semimartingale with values in $M$, $\alpha$ is section of $TM^*$, and $b$ is a section of $T^{(2,0)}M$. 

In the following we present the basic stochastic integrals on manifolds and some useful formulas (see M. Emery \cite{emery1}, P. Meyer \cite{meyer1}  and E. Hsu \cite{hsu1})

We denote by $\int \alpha\; \delta X$ the
Stratonovich integral of $\alpha$ along $X$, by $\int \alpha\;
d^{\nabla} X$ the It\^o integral of $\alpha$ along $X$, and by $\int b \;d(X,X)$ the
quadratic integral of $b$ along $X$.

The Stratonovich-It\^o conversion formula is given by:
\begin{equation}\label{conversion}
\int_0^t\alpha \delta X=\int_0^t\alpha d^{\nabla}
X+\frac{1}{2}\int_0^t\nabla\alpha\;(dX,dX).
\end{equation}

Let $M$ and $N$ be manifolds, and let $\nabla$ and $\nabla^{\prime}$ be connections of $M$ and $N$, respectively. Assume that $\alpha$ is a section of $TN^*$, $b$ is a section of $T^{(2,0)}N$, and $F:M\rightarrow N$ is a
smooth map. We have the following It\^ o formulas:
\begin{equation}\label{stratonovich-ito}
\int_0^t\alpha \;\delta F(X)=\int_0^tF^*\alpha \;\delta X
\end{equation}
,
\begin{equation}\label{quadratic-ito}
\int_0^tb\;(dF(X),dF(X))=\int_0^tF^*b\;(dX,dX).
\end{equation}
and
\begin{equation}\label{ito-ito}
\int_0^t\alpha \;d^{\nabla^{\prime}} F(X)=\int_0^tF^*\alpha
\;d^{\nabla} X +\frac{1}{2}\int_0^t\beta_F^*\alpha \;(dX,dX),
\end{equation}
where $\beta_F$ is the second fundamental form of $F$ (see
\cite{cat} for more details).

We recall that $X$ is a
$\nabla$-martingale if and only if $\int \alpha \;d^ {\nabla} X$
is a local martingale for any $\alpha \in \Gamma(TM^*)$.  It follows that $F$ is an affine
map if and only if  sends $\nabla$-martingales to
$\nabla^{\prime}$-martingales.

We assume that $M$ is a Riemannian manifold equipped with a metric $g$. The Brownian motion on $M$ is usually introduced has a diffusion generated by the half of the Laplace-Beltrami operator, but is convenient for us to consider the following equivalent definition of the Brownian motion on $M$ (see for instance M. Emery \cite{emery1} and P. Meyer \cite{meyer1}).  Let
$B$ be a continuous semimartingale with values in $M$, we say that
$B$ is a $g$-Brownian motion in $M$ if $B$ is a martingale with
respect to the Levi-Civita connection of $g$ and for any section
$b$ of $T^{(2,0)}M$ we have that
\begin{equation}\label{Brownian}
\int_0^tb(dB,dB)=\int_0^ttr\;b_{B_s}ds.
\end{equation}
Combining (\ref{stratonovich-ito}) with the property (\ref{Brownian}) we have
\begin{equation}\label{manabe}
\int_0^t\alpha \delta B=\int_0^t\alpha d^{\nabla}
B+\frac{1}{2}\int_0^td^*\alpha_{B_s}ds.
\end{equation}
From (\ref{ito-ito}) and (\ref{Brownian}) we get
\begin{equation}\label{harmonic}
\int_0^t\alpha d^{\nabla^{\prime}} F(B)=\int_0^tF^*\alpha
d^{\nabla} B +\frac{1}{2}\int_0^t\tau_F^*\alpha_{B_s} ds,
\end{equation}
where $\tau_F$ is the tension field of $F$.

We conclude this short introduction to the stochastic calculus on manifolds with the following characterization of harmonic maps. Let $F:M \rightarrow N$ be a differentiable map. Then $F$ is a 
harmonic map, this is $\tau_F=0$, if and
only if $F$ sends Brownian motions to
$\nabla^{\prime}$-martingales.

\section{Martingales on principal fiber bundles}
\noindent In this section we prove our main results.

\begin{theorem}\label{teo1}
Let $P(M,G)$ be a principal fiber bundle, and let $\omega$ be a connection
form on $P(M,G)$. Assume that $\nabla^{P}$ is a projectable connection with
projection $\nabla^{M}$ and that $Y$ is a continuous semimartingale
with values in $P$. Then $Y$ is a $\nabla^{P}$-martingale if and
only if
\begin{equation}\label{te1}
\int \omega \;\delta Y - \frac{1}{2}\int (\nabla^{P}\omega)(dY,dY)
\end{equation}
is a local martingale and
\begin{equation}\label{te2}
\int \alpha \;d^{\nabla^{M}} \pi \circ Y + \frac{1}{2} \int \alpha
\circ \pi_* \circ (2A^S+T^S)(dY,dY)
\end{equation}
is a local martingale for all $\alpha \in \Gamma(T^*M)$.
\end{theorem}
\begin{proof}
Let $Y$ be a $\nabla^{P}$-martingale. By the conversion formula (\ref{conversion}), we have
\[
\int \omega \;\delta Y = \int \omega d^ {\nabla^{P}} Y +
\frac{1}{2}\int (\nabla^{P}\omega)(dY,dY).
\]

\noindent Since $\int \omega \;d^ {\nabla^{P}} Y$ is a local
martingale, so is $\int \omega \;\delta Y-\frac{1}{2}\int
(\nabla^{P}\omega)(dY,dY)$. In order to
prove (\ref{te2}), we take $\alpha \in \Gamma(T^*M)$. It is easy
to check that
\begin{equation}\label{sff}
\beta_{\pi}^*\alpha = \alpha \circ \beta_{\pi} = -\alpha \circ
\pi_* \circ (2A^S + T^S).
\end{equation}
Combining the above identity and It\^o formula we conclude that
\[
\int \alpha \;d^{\nabla^{M}} \pi \circ Y + \frac{1}{2} \int \alpha
\circ \pi_* \circ (2A^S+T^S)(dY,dY)
\]
is the local martingale  $\int \pi^* \alpha\;d^{\nabla^{P}} Y$.

Conversely, take $\eta$ in $\Gamma(T^*P)$. Since the
$\mathcal{C}^{\infty}$-module $\Gamma(T^*P)$ is generated by
$\omega$ and by the differential forms $\pi^* \alpha$ with $\alpha
\in \Gamma(T^*M)$, we have that $\eta$ is a linear combination of
 differential forms $f\pi^* \alpha$ and $h\omega$ with $f,h \in
\mathcal{C}^{\infty}(P)$. It is clear that $\int h\omega \;
d^{\nabla^P}Y=\int h(Y) \; d(\int \omega \; d^{\nabla^P}Y)$ is a
local martingale and that
\[
\int f \pi^* \alpha \; d^{\nabla^P}Y=\int f(Y) \; d(\int \pi^*
\alpha \; d^{\nabla^P}Y).
\]
Hence, in order to show that $\int \eta \; d^{\nabla^P}Y$ is a
local martingale, it is sufficient to show that $\int \pi^* \alpha
\; d^{\nabla^P}Y$ is a local martingale. Applying the It\^o
formula (\ref{ito-ito}) and (\ref{sff}) we deduce that

\begin{eqnarray*}
\int \pi^* \alpha \; d^{\nabla^P}Y & = & \int \alpha \;
d^{\nabla^M}\pi \circ Y - \frac{1}{2} \int \beta_{\pi}^* \alpha
(dY, dY)\\
& = & \int \alpha \;d^{\nabla^{M}} \pi \circ Y + \frac{1}{2} \int
\alpha \circ \pi_* \circ (2A^S+T^S)(dY,dY).
\end{eqnarray*}
This completes the proof.

\end{proof} \eop

\begin{theorem}\label{teo2} Let $P(M,G)$ be a principal fiber bundle, and let $\omega$ be a connection
form on $P(M,G)$. Assume that $\nabla^{P}$ is a projectable connection with
projection $\nabla^{M}$ and that $N$ is a Riemannian manifold with
metric $g$. Furthermore, let $F:N \rightarrow P$ be a smooth map. Then $F$ is a
harmonic map if and only if
\begin{equation}
d^{*}F^{*} \omega = trF^{*}(\nabla^{P}\omega).
\end{equation}
and
\begin{equation}
\tau_{\pi \circ F} = -\mathrm{tr}~ \pi_* \circ (2A^S+T^S) \circ
F_* \otimes F_*.
\end{equation}
\end{theorem}

\begin{proof}
\noindent Let $F$ be an harmonic map and $B$ be a $g$-Brownian
motion. From the Bismut characterization of harmonic maps and
Theorem \ref{teo1} we see that
\begin{equation}\label{eq1}
 \int \omega \delta F(B) - \frac{1}{2}\int (\nabla^{P}\omega)\;(dF(B),dF(B))
\end{equation}
is a local martingale. Applying (\ref{stratonovich-ito}) and (\ref{quadratic-ito}) we can rewrite (\ref{eq1}) as
\begin{equation}\label{eq2}
\int F^{*}\omega \delta B - \frac{1}{2}\int
F^{*}(\nabla^{P}\omega)(dB,dB).
\end{equation}
From the Manabe formula (\ref{manabe}) we have
\begin{equation}\label{eq3}
\int F^*\omega \delta B =  \int F^{*} \omega d^{\nabla^{g}}B+
\frac{1}{2}\int d^{*}F^{*}\omega_{B_{s}}\; ds,
\end{equation}
$\nabla^{g}$ being the Levi-Civita connection associated to
$g$. Combining (\ref{eq2}) and (\ref{eq3}) we conclude that
\begin{equation}\label{eq4}
\int \omega d^{\nabla^{P}} F(B) + \frac{1}{2}\int
(d^{*}F^{*}\omega - tr F^{*}(\nabla^{P}\omega))_{B_{s}}\; ds
\end{equation}
is a local martingale. Doob-Meyer decomposition now yields
\[
\int (d^{*}F^{*} \omega - tr F^{*}(\nabla^{P}\omega))_{B_s}\;
ds=0.
\]
Since $B$ is arbitrary, it follows that $d^{*}F^{*} \omega =
trF^{*}(\nabla^{P}\omega)$.

It remains to prove that $\tau_{\pi \circ F} = -\mathrm{tr}~ \pi_*
\circ (2A^S+T^S) \circ F_* \otimes F_*$. As $F$ is an harmonic
map, it follows that $F(B)$ is a $\nabla^{P}$-martingale. From
Theorem \ref{teo1}, we see that
\[
\int \alpha \;d^{\nabla^{M}} \pi \circ F(B) + \frac{1}{2} \int
\alpha \circ \pi_* \circ (2A^S+T^S)(dF(B),dF(B))
\]
is a local martingale for all $\alpha \in \Gamma(T^*M)$. Applying (\ref{Brownian}) and (\ref{harmonic}) we conclude that
\[
\int (\pi \circ F)^*\alpha \;d^{\nabla^{g}} B + \frac{1}{2} \int
(\tau_{\pi \circ F}^* \alpha + \mathrm{tr}~\alpha \circ \pi_*
\circ (2A^S+T^S)\circ F_* \otimes F_*)_{B_s}ds
\]
is a local martingale. Since $B$ and $\alpha \in \Gamma(T^*M)$ are
arbitrary, we have
\[
\tau_{\pi \circ F} = -\mathrm{tr}~ \pi_* \circ (2A^S+T^S) \circ
F_* \otimes F_*.
\]

Conversely, suppose that $B$ is a $g$-Brownian motion. From the
Bismut characterization is sufficient to show that $F(B)$ is a
$\nabla^P$-martingale. From what proved, it can be written
$ \int \omega  \delta F(B)- \frac{1}{2}\int (\nabla^{P}\omega)\;(dF(B),dF(B))$
as
\[
\int F^* \omega d^{\nabla^g}B+ \frac{1}{2}\int (d^*F^* \omega -
\mathrm{tr} F^*(\nabla^{P}\omega))_{B_s}\; ds.
\]
Since $d^*F^* \omega = \mathrm{tr} F^*(\nabla^{P}\omega)$, it
follows that
\[
\int \omega \delta F(B) - \frac{1}{2}\int
(\nabla^{P}\omega)\;(dF(B),dF(B))
\]
is a local martingale. It remains to prove that
\begin{equation}\label{semi}
\int \alpha \;d^{\nabla^{M}} \pi \circ F(B) - \frac{1}{2} \int
\alpha \circ \pi_* \circ (A+T)(dF(B),dF(B))
\end{equation}
is a local martingale for all $\alpha \in
\Gamma(T^*M)$. An straightforward calculation shows that
we can rewrite the semimartingale (\ref{semi}) as
\[
\int (\pi \circ F)^*\alpha \;d^{\nabla^{g}} \pi \circ B
+\frac{1}{2}\int (\tau_{\pi \circ F}^*\alpha - \alpha \circ
\mathrm{tr}~ \pi_* \circ (A+T) \circ F_* \otimes F_*)_{B_s} ds,
\]
and so (\ref{semi}) is a local martingale.
Therefore $F(B)$ is an $\nabla^P$-martingale by Theorem
\ref{teo1}. \eop
\end{proof}

\begin{example}
Let $M$ be a differentiable manifold, and let $\nabla$ be a connection
on $M$. We consider the frame bundle $BM(M,GL(\mathbb{R}^n))$
which is endowed with $\nabla^c$ and $\nabla^h$, the canonical lift and
horizontal lift of $\nabla$, respectively. Let $\omega$ be the
connection form on $BM$ which is associated with $\nabla$. The following assertions are true.

1) $T=0$ and $\pi_* \circ A=0$ for $\nabla^c$ and $\nabla^h$.

2) The symmetric part of $\nabla^h\omega$ is $-\omega \odot
\omega$.

3) The symmetric part of $\nabla^c\omega$ is $-\omega \odot \omega
+a^c$ (see \cite{cat-stel} for the definition of $a^c$).

Applying Theorem \ref{teo1} and Theorem \ref{teo2} we
recovering the main results of \cite{cat-stel}. A $BM$-valued
semimartingale $Y$ is a $\nabla^{h}$-martingale (
$\nabla^{c}$-martingale) if and only if $\pi \circ Y$ is a $\nabla$-martingale in $M$ and
\[
\int \omega \;\delta Y+\frac{1}{2}\int (\omega \odot \omega)(dY,dY)
\]
is martingale local ($\pi \circ Y$ is a $\nabla^{M}$-martingale in $M$ and
\[
\int \omega \;\delta
Y+\frac{1}{2}\int\omega\odot\omega(dY,dY)+\frac{1}{2}\int
a^c\;(dY,dY)
\]
is martingale local). Furthermore, $F:N \rightarrow BM$ is
$(g,\nabla^{h})$-harmonic map ($(g,\nabla^{c})$-harmonic map) if
and only if $\pi \circ F$ is a $(g,\nabla)$-harmonic map and $d^{*}F^{*}
\omega+ trF^{*}(\omega \odot\omega)=0$ ($\pi \circ F$ is a $(g,\nabla)$-harmonic map  and $d^{*}F^{*}
\omega+ trF^{*}(\omega \odot\omega +a^c)=0$).

\end{example}

\begin{corollary}\label{coro1}
Let $P(M,G)$ be a principal fiber bundle, and let $\omega$ be a connection
form on $P$. Assume that $k_0$ is the $ad(G)$-invariant metric on $\mathfrak{g}$
and that $h$ is a Riemannian metric on $M$. We consider the
$G$-invariant Riemannian metric $k$ on $P$ defined by
\begin{equation}\label{KK2}
	k(U,V)=h(\pi_*(U),\pi_*(V))+k_0(\omega(U), \omega(V))
\end{equation}
for all $U$ and $V$ vector fields on $P$. Let us denote by $\nabla^k$ the
Riemannian connection associated to $k$ and by $\nabla$ the one associated to $h$. We have the following
assertions:
\begin{enumerate}
\item A $P$-valued semimartingale $Y$ is a $\nabla^{k}$-martingale
if and only if

1) $\int \omega \;\delta Y$ is a local martingale; and,

2) $\int \alpha \;d^{\nabla} \pi \circ Y - \frac{1}{2} \int \alpha
\circ \pi_* \circ A(dY,dY)$ is a local martingale.

\item Let $N$ be a Riemannian manifold with metric $\bar{g}$. A smooth map $F:N \rightarrow P(M,G)$ is  a
 $(\bar{g},\nabla^{P})$-harmonic map if and only if

1) $d^{*}F^{*} \omega=0$; and,

2) $\tau_{\pi \circ F} = \mathrm{tr}~ \pi_* \circ A \circ F_*
\otimes F_*$.

\end{enumerate}
\end{corollary}

\begin{proof}
Since $T=0$, it is sufficient to show that the symmetric part of
$\nabla^k \omega$ is zero. From (\ref{grupo2}) an easy
calculations shows that
\begin{displaymath}
\begin{array}{ccccl}
\nabla^k \omega (B^*,X^h) & = & -\omega(\nabla^k_{B^*}X^h) & = & 0 \\
\nabla^k \omega (X^h,B^*) & = & -\omega(\nabla^k_{X^h}B^*) & = & 0 \\
\nabla^k \omega (B^*,C^*) & = & -\omega(\nabla^k_{B^*}C^*) & = & -\frac{1}{2}[B,C] \\
\nabla^k \omega (X^h,Y^h)& = & -\omega(\nabla^k_{X^h}Y^h) & = &
-\omega(A_{X^h}Y^h).
\end{array}
\end{displaymath}
According to (\ref{KK2}), we have $A_{Z^h}Z^h=0$ for all $Z$
vector field on $M$. It follows that the symmetric part of
$\nabla^k \omega$ is zero. \eop
\end{proof}

\begin{remark}
M. Arnaudon and S. Paycha, in \cite{arnaudon-paycha}, shows that
semimartingales in a principal fiber bundle $P(M,G)$ with $G$-invariant Riemannian me-tric $k$ can be decomposed into $G$-
and $M$- valued semimartingales. More precisely, a semimartingale
$Y$ with values in $P(M,G)$ splits in a unique way into a
horizontal semimartingale $\widetilde{Y}$ and a semimartingale $V$
with values in $G$ such that
\[
Y= \widetilde{Y} \cdot V.
\]
Moreover, $V$ is the stochastic exponential
\[
V = \epsilon(\int \omega \delta Y)
\]
and $\widetilde{Y}$ is the solution of the It\^o equation
\[
d^{\nabla^k}\widetilde{Y}=H^k_{\widetilde{Y}}d^{\nabla}(\pi \circ
Y).
\]
It follows that $\widetilde{Y}$ is a $\nabla^{k}$-martingale if
and only if $\pi \circ Y$ is a $\nabla$-martingale. In the case
that $k$ is given by (\ref{KK2}), Corollary \ref{coro1} shows that
if $Y$ is a $\nabla^{k}$-martingale then $V$ is a $G$-martingale.

Finally, we consider $Y$ a solution of the It\^o
equation
\[
d^{\nabla^k}Y = \sum_{i=1}^nE_i(Y)dB^i,
\]
where $(B^1,...,B^n)$ is a Brownian motion in $\mathbb{R}^n$ and
the $E_i$ are vertical or horizontal vector fields on $P$. It is
clear that $Y$ is a $\nabla^{k}$-martingale and follows easily
that $\widetilde{Y}$ is a $\nabla^{k}$-martingale and $V$ is a
$G$-martingale.

\noindent
\end{remark}

\end{document}